\documentclass[a4paper,11pt,DIV=11,oneside,headings=small,bibliography=totoc,abstract=true,intlimits]{scrartcl}
\usepackage{amssymb,amsmath,amsfonts,amsthm}
\allowdisplaybreaks
\usepackage[english]{babel}
\usepackage{csquotes}
\usepackage[T1]{fontenc}
\usepackage{enumerate}
\usepackage[blocks]{authblk}

\usepackage[backend=biber,style=alphabetic,maxbibnames=20,maxalphanames=4,giveninits=true]{biblatex}
\usepackage[linktocpage, pdftex,colorlinks]{hyperref}
	\usepackage{color}
 	\definecolor{darkred}{rgb}{0.5,0,0}
 	\definecolor{darkgreen}{rgb}{0,0.5,0}
 	\definecolor{darkblue}{rgb}{0,0,0.5}
 	\hypersetup{colorlinks,linkcolor=darkblue,filecolor=darkgreen,urlcolor=darkred,citecolor=darkblue}
    \hypersetup{
    pdftitle={},
    pdfauthor={},
}
\newcommand{\drm}{\mathrm{d}}
\newcommand{\euler}{\mathrm{e}}
\newcommand{\Z}{\mathbb{Z}}
\newcommand{\N}{\mathbb{N}}
\newcommand{\R}{\mathbb{R}}
\newcommand{\EE}{\mathbb{E}}
\newcommand{\PP}{\mathbb{P}}
\newcommand{\C}{\mathbb{C}}
\newcommand{\1}{\mathbf{1}}
\newcommand{\from}{\colon}
\newcommand{\norm}[1]{\left\lVert#1\right\rVert}
\newcommand{\abs}[1]{\left\lvert#1\right\rvert}
\newcommand{\Deg}{\operatorname{Deg}}
\newcommand{\Dmax}{D_{\mathrm{max}}}
\newcommand{\id}{\operatorname{Id}}
\DeclareMathOperator{\Covr}{Covr}
\DeclareMathOperator{\Inr}{Inr}
\DeclareMathOperator{\vol}{vol}
\newcommand{\comb}{\mathrm{comb}}
\newcommand{\Le}{{L}}
\renewcommand{\epsilon}{\varepsilon}
\newtheorem{theorem}{Theorem}[section]
\newtheorem{proposition}[theorem]{Proposition}
\newtheorem{lemma}[theorem]{Lemma}
\newtheorem{corollary}[theorem]{Corollary}
\theoremstyle{definition}

\newtheorem{definition}[theorem]{Definition}
\theoremstyle{remark}
\newtheorem{remark}[theorem]{Remark}
\title{On controllability, observability and stabilizability of the heat equation on discrete graphs}
\author[1]{Florentin M\"unch}
\affil[1]{Universit\"at Leipzig, Mathematisches Institut,
Augustusplatz 10,
04109 Leipzig, Germany
}
\author[2,3]{Christian Seifert}

\affil[2]{Technische Universit\"at Hamburg, Institut f\"ur Mathematik, 
Am Schwarzenberg-Campus 3, 21073 Hamburg, 
Germany
}
\affil[3]{University of the Free State,
Mathematics and Applied Mathematics,
Bloemfontein 9300, Republic of South Africa
} 

\author[4]{Peter Stollmann}
\affil[4]{Technische Universit\"at Chemnitz, Fakult\"at f\"ur Mathematik, 
Reichenhainer Stra{\ss}e 41, 09126 Chemnitz, 
Germany
}
\author[1]{Martin Tautenhahn}

\date{\vspace{-7ex}}
\begin{document}
\maketitle
\begin{abstract}
We consider linear control problems for the heat equation of the form $\dot f (t) = -Hf (t) + \1_D u (t)$, $f (0) \in \ell_2 (X,m)$, where $H$ is the weighted Laplacian on a discrete graph $(X,b,m)$, and where $D \subseteq X$ is relatively dense.
We show cost-uniform $\alpha$-controllability by means of a weak observability estimate for the corresponding dual observation problem. 
We discuss optimality of our result as well as
consequences on stabilizability properties. 
\\[1ex]
\textsf{\textbf{Mathematics Subject Classification (2020).}} 93B05, 93B07, 05C63, 35K05, 81S07
\\[1ex]
\textbf{\textsf{Keywords.}} controllability, observability, stabilizability, weighted graph, Laplacian, uncertainty relation 
\end{abstract}
\tableofcontents
\section{Introduction}
In recent years, discrete graphs have been studied extensively from various perspectives, e.g.\ from their geometry, as ground spaces for partial differential equations, as state spaces for stochastic processes, as well as spatial spaces for applications in mathematical physics, see \cite{KellerL-12,KellerLW-21,LenzSS-20} and references therein for an overview on the topic. Moreover, discrete graph models also arise in discretizations of continuum models, and their interplay was studied e.g.\ in \cite{PostS-21, KostenkoN-22, KostenkoN-23}. 

In this note, we add a further study to this line of research, in particular, we consider linear time-invariant control systems as well as observation systems for the heat equation on discrete graphs and study controllability, observability, as well as stabilizabilty properties. Here, the time evolution of the heat equation is driven by the weighted Laplacian $H$ which is the generator of the corresponding solution operator semigroup. More precisely, for a given weighted graph $(X,b,m)$ let $H$ be the weighted Laplacian in $\ell_2(X,m)$ given by
\[Hf(x) := \frac{1}{m(x)} \sum_{y\in X} b(x,y) \bigl(f(x)-f(y)\bigr),\]
which is a non-negative self-adjoint operator.
\par
For $D\subseteq X$ we consider the linear control problem
\[\dot{f}(t) = -H f(t) + \1_D u(t), \quad t>0,\quad f(0)=f_0\in \ell_2(X,m),\]
where $\1_D$ is the canonical embedding and $u$ attains values in $\ell_2(D,m|_D)$. We aim to study cost-uniform $\alpha$-controllability in time $T > 0$ with $\alpha \geq 0$, which means that for all inital states $f_0 \in\ell_2(X,m)$ we can find a \emph{control function} $u \in L_r ((0,T); \ell_2(D,m|_D))$, $r \in [1,\infty]$, such that 
\begin{equation} \label{eq:control-intro}
 \lVert u \rVert_{L_r((0,T);\ell_2 (D,m|_D))} \le K \lVert f_0 \rVert_{\ell_2 (X,m)} \quad \text{and}\quad \lVert f(T) \rVert_{\ell_2 (X,m)} \leq \alpha \lVert f_0 \rVert_{\ell_2 (X,m)}.
\end{equation}
That is, the norm of the solution at time $T$ is at most an $\alpha$-portion of the norm of $f_0$. The constant $K$ is interpreted as the \emph{control cost}. Of course, given $T>0$ one aims to prove \eqref{eq:control-intro} with $\alpha$ and $K$ as small as possible, optimally with $\alpha = 0$.

In order to prove $\alpha$-controllability we make use of the dual observation problem
\begin{align*}
    \dot{\varphi}(t) & = -H \varphi(t), \quad t>0,\quad \varphi(0)=\varphi_0 \in \ell_2(X,m),\\
    \psi(t) & = \varphi(t)|_D,\quad t\geq 0.
\end{align*}
and prove a so-called \emph{weak observability estimate} for that observation problem.
A classical method to study observability is given by the Lebeau--Robbiano method \cite{LebeauR-95, Miller-10}, where an uncertainty principle (also called a spectral inequality and strongly related to unique continuation) together with a dissipation estimate yields a final-time observability estimate for an observation system, and thus by means of duality, also controllability for the corresponding control problem. This strategy has been applied successfully for various model, see e.g.\ \cite{LebeauZ-98, TenenbaumT-11, WangZ-17, BeauchardP-18, NakicTTV-20, GallaunST-20, BombachGST-23, TrelatWX-20, BeauchardJPS-21, AlphonseM-22, LiuWXY-22, BombachGST-23a, KruseS-23, GallaunMS-23, EgidiGST-24}.
While the dissipation estimate is an easy consequence of spectral theory in the case of non-negative self-adjoint generators $H$, the uncertainty principle is more involved. Fortunately, such uncertainty principles, also in a quantitative version, have already been investigated for relatively dense sets $D$ in the range of small energies, see \cite{LenzSS-20}. We will make use of these to show a weak obervability estimate for final time $T>0$; that is, we show that we can bound the norm of $\varphi(T)$ by an $L_r$-norm of $\psi$ and a portion of the norm of $\varphi_0$. In turn, we conclude cost-uniform $\alpha$-controllability for given $\alpha>0$ and sufficiently large final times $T>0$ by duality. Let us note that the standard method proving observability from an uncertainty principle does not apply as it assumes uncertainty at all energies. On discrete graphs uncertainty principles are (currently) available only at small energies.
\par
Moreover, we provide concrete examples for which $0$-controllability fails, as well as a method to generate such examples. This, in turn, also suggests that one cannot expect that the quantitative uncertainty principles for the weighted Laplacian in \cite{LenzSS-20} can be extended in their present form to all energies. This observation aligns well to the fact that unique continuation for solutions of the heat equation fails on discrete graphs. Let us note that this complements the situation in the continuum, where $0$-controllability can be shown \cite{BeauchardP-18,EgidiV-18,NakicTTV-20,GallaunST-20,BombachGST-23}. In fact, in the continuum case, \cite{EgidiV-18} (see also \cite{BombachGST-23a}) has shown that $0$-controllability is equivlaent to so-called thickness of the control set $D$. Thus, the discrete situation is different to the continuum case: in general, we cannot obtain $0$-controllability in the discrete case.

We also comment on the necessity of the relative denseness of $D$ for $\alpha$-controllability, which, again, can be seen as a discrete counterpart to the necessary conditions in \cite{EgidiV-18, BombachGST-23a}.

Lastly, we derive stabilizability for the control problem on discrete graphs, in a standard manner, from $\alpha$-controllabilty for $\alpha<1$.

The paper is organised as follows. In Section \ref{sec:weighted_graphs} we introduce the relevant notation on weighted graphs and the weighted Laplacian, as well as covering graphs. We formulate our main results in Section \ref{sec:main_results}, split into those on controllability and weak observability, non-$0$-controllability, necessary conditions for controllability, and consequences for exponential stabilization. Section \ref{sec:proof_obs_control} is devoted to the proofs of the main results.

\section{Weighted graphs and Laplacians}
\label{sec:weighted_graphs}
In this section, we introduce the notation on weighted graphs and Laplacians on graphs.

\paragraph{Graph theoretic notation}
Let $X$ be a countable set and $m \from X \to (0,\infty)$ a weight function, which induces a measure on $X$, also denoted by $m$, via $m(A):=\sum_{x\in A}m(x)$ for all $A\subseteq X$. We now choose a symmetric weight function $b\from X\times X\to [0,\infty)$ such that $b(x,x) = 0$ for all $x\in X$ and $$\sum_{y\in X} b(x,y) < \infty$$ for all $x\in X$. 
We then say that $(X,b,m)$ is a \emph{weighted graph}. We refer, e.g., to the survey \cite{Keller-15} or the textbook \cite{KellerLW-21} and the references therein for more details on weighted graphs. 
A \emph{path} in $X$ is a finite sequence $\gamma = (x_0,x_1,\ldots,x_k)$ with $x_j\in X$ such that $b(x_j,x_{j+1})>0$ for all $j\in\{0,\ldots,k-1\}$; specifying the endpoints, we say that $\gamma$ is a \emph{path from $x_0$ to $x_k$}.
We say that the weighted graph $(X,b,m)$ is \emph{connected} if for all $x,y\in X$ with $x\neq y$ there exists a path from $x$ to $y$. Throughout the paper we make the following assumptions:
\begin{enumerate}[(A)]
\setcounter{enumi}{2}
 \item \label{eq:C} $(X,b,m)$ is connected,
 \setcounter{enumi}{1}
 \item \label{eq:B} $\Dmax := \displaystyle{\sup_{x\in X} \Deg (x)<\infty}$, where $\displaystyle{\Deg(x) := \frac{1}{m(x)} \sum_{y\in X} b(x,y)}$,
  \setcounter{enumi}{12}
 \item  \label{eq:M}$\displaystyle{\sup_{x\in X} m(x) <\infty}$.
\end{enumerate}

\paragraph{Metrics and related notions on graphs}

The first metric we consider is the \emph{combinatorial metric} $d_{\comb}\from X\times X\to [0,\infty)$ which is defined as
\[d_{\comb}(x,y) := \begin{cases}
                           0 & x=y,\\
                           \min\{n\in\N:\; \text{$\exists$ path $\gamma=(x_0,x_1,\ldots,x_n)$ with $x_0 = x$ and $x_n=y$}\} & x\neq y.
                          \end{cases}\]
Clearly, $d_{\comb}$ is a metric and measures the distance of two points by means of the combinatorial graph neglecting the particular values of the edge weights describes by $b$.

The second metric we consider is the so-called length metric.
For a path $\gamma= (x_0,x_1,\ldots,\allowbreak x_k)$ in $X$ we define its \emph{length} as
    \[L(\gamma):= \sum_{j=0}^{k-1} \frac{1}{b(x_j,x_{j+1})}.\]
We define the \emph{length distance} $d_{\Le}\from X\times X\to [0,\infty)$ by
    \[d_{\Le}(x,y):=\begin{cases} 0 & x=y,\\
               \inf\{L(\gamma):\; \gamma \text{ path from $x$ to $y$}\} & x\neq y.
              \end{cases}\]
Note that $d_{\Le}$ is finite by \eqref{eq:C}, and by \eqref{eq:B} and \eqref{eq:M} we have 
\[
\sup_{x,y\in X} b(x,y) \leq \sup_{x\in X} \left( m(x) \Deg (x) \right) \leq \left(\sup_{x\in X} m(x) \right) \left(\sup_{x \in X} \Deg (x) \right) < \infty .
\]
Hence, $b$ is bounded, and therefore 
\[
d_{\Le}(x,y) \geq \frac{1}{\sup_{x,y\in X} b(x,y)}>0 
\]
for all $x,y\in X$ with $x\neq y$, so $d_{\Le}$ is a metric. 


By \eqref{eq:B} and \eqref{eq:M}, for a path $\gamma=(x_0,x_1,\ldots,x_k)$ we have
$L(\gamma) \geq k / \sup_{x,y\in X} b(x,y)$.
This implies
\[d_{\comb} \leq \sup_{x,y\in X} b(x,y) d_{\Le}.\]
If aditionally $\inf_{x,y\in X, b(x,y)>0} b(x,y)>0$, then we also obtain
\[d_{\Le} \leq \frac{1}{\inf_{x,y\in X, b(x,y)>0} b(x,y)} d_{\comb}\]
such that $d_{\Le}$ and $d_{\comb}$ are equivalent.
\par
Let $d$ be a metric on $X$.
For $x\in X$ and $r\geq 0$ we denote the open and closed ball of radius $r$ around $x$ by $U_r(x):=U_r^d(x):=\{y\in X:\; d(x,y)<r\}$ and $B_r(x):=B_r^d(x):=\{y\in X:\; d(x,y)\leq r\}$, respectively. For $D\subseteq X$ we define the \emph{covering radius} of $D$ by
    \[\Covr(D):=\Covr^d(D) := \inf\Bigl\{R>0:\; \bigcup_{x\in D} B_R(x) = X\Bigr\}\in [0,\infty]\]
    and say that $D$ is \emph{$d$-relatively dense} if $\Covr(D)<\infty$. 
    Moreover, for $\Omega\subseteq X$ we define the \emph{inradius} by
    \[\Inr(\Omega):=\Inr^d(\Omega) := \sup\{r>0:\; \text{there exists $x\in \Omega$ such that $U_r(x)\subseteq \Omega$}\}\in [0,\infty].\]
Finally, for $r\geq 0$ we define the \emph{maximum volume of balls} of radius $r$ by
    \[\vol(r):=\vol^d(r):=\sup_{x\in X} m(B_r(x))<\infty.\] 
Usually, we will supress the superscript $d$, whenever it is clear from the context which of the two metrices $d_\Le$ or $d_\comb$ is considered.
Note that \eqref{eq:B} and \eqref{eq:M} imply that balls of fixed finite radius with respect to $d_\Le$ have uniformly bounded measure; cf.\ \cite[Proposition~2.1]{LenzSS-20}.
\paragraph{Laplace operators on graphs}
Consider the Hilbert space $$\ell_2 (X,m) := \{ f : X \to \C \colon \sum_{x \in X} \lvert f (x) \rvert^2 m (x) < \infty \}$$ equiped with the inner product
\[
 \langle f,g \rangle := \sum_{x \in X} f(x) \overline{g(x)} m(x) .
\]
Let $H$ be the bounded self-adjoint operator $H\colon \ell_2(X,m)\to \ell_2(X,m)$, known as the \emph{weighted Laplacian}, given by
\[Hf(x) := \frac{1}{m(x)} \sum_{y\in X} b(x,y) \bigl(f(x)-f(y)\bigr).\]
Note that $H$ is bounded by \eqref{eq:B}.
We denote by $(S_t)_{t\geq 0}$ the $C_0$-semigroup generated by $-H$, i.e.\ $S_t:=\euler^{-tH}$ for $t\geq 0$. Note that $(S_t)_{t\geq 0}$ is contractive as $H$ is self-adjoint and non-negative.

\paragraph{Covering graphs}

\begin{definition}
Let $(X_1,b_1,m_1)$ and $(X_2, b_2, m_2)$ be two weighted graphs. Then $(X_2, b_2, m_2)$ is a \emph{covering graph} of $(X_1,b_1,m_1)$ if $(X_2,b_2,m_2)$ is connected and there exists a surjective map $p: X_2\to X_1$ such that for all $x_2 , y_2 \in X_2$ we have
\begin{itemize}
\item $m_1(p(x_2)) = m_2(x_2)$,
\item $b_2(x_2,y_2) = b_1(p(x_2),p(y_2))$,
\item and that $p$ maps the set of neighbors $\{y_2\in X_2:\; b_2(x_2,y_2)>0\}$ of $x_2$ bijectively onto the set of neighbors $\{y_1\in X_1:\; b_1(p(x_2),y_1)>0\}$ of $p(x_2)$.
\end{itemize}
\end{definition}

\begin{remark}
    If $(X_1,b_1,m_1)$ is a finite graph and $(X_2, b_2, m_2)$ a covering graph of $(X_1,b_1,m_1)$ then $(X_2, b_2, m_2)$ has bounded geometry, i.e.\ $\infty > \sup m_2 \geq \inf m_2>0$, and $\infty > \sup b_2 \geq \inf (\{ b_2(x_2, y_2):\; x_2,y_2\in X_2\}\setminus\{0\}) >0$, and the number of neighbors of each vertex in $X_2$ is uniformly bounded. In particular, $(X_2, b_2, m_2)$ satisfies assumptions \eqref{eq:B} and \eqref{eq:M}, and clearly also \eqref{eq:C}.
\end{remark}

\begin{definition}
A weighted graph $(X,b,m)$ is called \emph{amenable} if there exists a sequence $(Y_n)_{n\in\N}$ of finite subsets of $X$ such that
$Y_n \subseteq Y_{n+1}$ for all $n \in \N$ and $\bigcup_{n\in \N} Y_n = X$, and
\[
\frac{b(Y_n,Y_n^c)}{m(Y_n)} \to 0 \mbox{ as }n \to \infty,
\]
where $b(Y_n,Y_n^c):=\sum_{x\in Y_n,y\in Y_n^c} b(x,y)$.
\end{definition}
Note that amenability is equivalent to $\inf \sigma(H)=0$; cf.\ \cite[Theorem 13.2 and Exercise 13.2]{KellerLW-21}.

\section{Main results}
\label{sec:main_results}

In this section we state our main results.

\subsection{Controllability and weak observability}

Let $D \subseteq X$ and $\1_D : \ell_2 (D,m|_D) \to \ell_2 (X,m)$ be the canonical embedding defined by $(\1_D u)(x) := u(x)$ if $x \in D$ and $(\1_D u)(x) := 0$ if $x \not\in D$. We consider the linear control problem $(H,D)$ given by
\begin{equation} \label{eq:system}
 \dot f (t) = -Hf (t) + \1_D u (t), \quad t > 0,  \quad f (0) = f_0 \in \ell_2 (X,m) ,\tag{$H,D$}
\end{equation}
where $u \in L_r ((0,\infty);\ell^2 (D,m|_D))$ for some $r \in [1,\infty]$ is the so-called \emph{control function}. Thus, control function $u$ influences the equation only via the subset $D$ of $X$. The mild solution to \eqref{eq:system} is for $t\geq 0$ given by Duhamel's formula
\[
 f (t) = S_t f_0 + \int_0^t S_{t-\tau} \1_D u (\tau) \drm \tau .
\]
We are mainly interested in cost-uniformly $\alpha$-controllability and its equivalent properties, see e.g.\ \cite{TrelatWX-20} for the case of Hilbert spaces, or \cite{EgidiGST-24} for the general framework of Banach spaces.
\begin{definition} \label{def:controllability}
Let $T > 0$, $\alpha,K \geq 0$, and $r \in [1,\infty]$.
 The linear control problem \eqref{eq:system} is called \emph{cost-uniformly $\alpha$-controllable in time $T$ with respect to $L_r$ and cost $K$} if for all $f_0\in \ell_2 (X,m)$ there exists $u \in L_r((0,T);\ell_2 (D,m|_D))$ such that 
	$$ \lVert u \rVert_{L_r((0,T);\ell_2 (D,m|_D))} \le K \lVert f_0 \rVert_{\ell_2 (X,m)} \quad \text{and}\quad \lVert f(T) \rVert_{\ell_2 (X,m)} \leq \alpha \lVert f_0 \rVert_{\ell_2 (X,m)}.$$
We will use the short hand notation $(\alpha,T,r,K)$-controllable for this property.
Note that we consider $L_r((0,T);\ell_2 (D,m|_D))$ as subspace of $L_r((0,\infty);\ell_2 (D,m|_D))$ via extension by zero.
\end{definition}

\begin{theorem}
\label{thm:main}
    Let $D\varsubsetneq X$ be $d_{\Le}$-relatively dense. Let $\alpha > 0$, $r \in [1,\infty]$.
    Then there exist $T>0$ and $K\geq 0$ such that the linear control problem \eqref{eq:system} is $(\alpha,T,r,K)$-controllable.
\end{theorem}

Theorem \ref{thm:main} is a consequence of the following theorem on weak observability and duality; for details see Section \ref{sec:proof_obs_control}.

\begin{theorem}
\label{thm:weak_obs}
    Assume that $D\varsubsetneq X$ is $d_{\Le}$-relatively dense and $\Omega:=X\setminus D$.
    Let $T>0$, $\delta\in [0,1)$ and $r\in [1,\infty]$. Then for all $\varphi_0\in \ell^2(X,m)$ we have
    \[
    \lVert S_T\varphi_0 \rVert 
    \leq 
    K \lVert (S_{(\cdot)} \varphi_0)|_D \rVert_{L_{r}((\delta T,T);\ell^2(D,m|_D))} + \alpha \norm{\varphi_0},
    \]
    where $2\lambda:= (\Inr(\Omega)\vol(\Inr(\Omega)))^{-1}$, $\kappa:=8 \norm{H+1}^2 \Inr(\Omega)\vol(\Inr(\Omega))$,
    \[
     K := \frac{\kappa}{((1-\delta) T)^{1/r}} , \quad\text{and}\quad
     \alpha := (\kappa+1)\euler^{-\delta\lambda T} .
    \]
\end{theorem}

If $D=X$, Theorem~\ref{thm:main} and Theorem~\ref{thm:weak_obs} remain true with $K = 1/((1-\delta) T)^{1/r}$ and $\alpha=0$.

\subsection{Non-\texorpdfstring{$0$}{0}-controllability}

It turns out that, in general, we cannot have (cost-uniform) $0$-controllability, so Theorem \ref{thm:main} with $\alpha=0$ is wrong. 
In this direction we have the following theorem which relates non-$0$-controllability for finite graphs to their covering graphs.

\begin{theorem}\label{thm:CoverNotControllable}
Let $(X_1,b_1,m_1)$ be a finite connected graph and $(X_2,b_2,m_2)$ an amenable covering graph of $(X_1,b_1,m_1)$ with weighted graph Laplacians $H_1$ and $H_2$, respectively, $D_1\subseteq X_1$. If the linear control system $(H_1,D_1)$ is not $(0,T,r,K)$-controllable for some $T>0$, $r\in [1,\infty]$, $K\geq 0$, then for all $T>0$, $r\in [1,\infty]$ and $K\geq 0$ the linear control problem $(H_2, p^{-1}(D_1))$ is not $(0,T,r,K)$-controllable.
\end{theorem}

One example where this theorem can be applied is given in the following theorem.

\begin{theorem}
\label{thm:non-null-controllability_Z}
Let $X=\Z$, $m=1$, $b(x,y) = 1$ if $|x-y|=1$ and zero otherwise, $D= 2\Z$, $T>0$, $r\in[1,\infty]$ and $K\geq 0$. Then the linear control problem \eqref{eq:system} is not $(0,T,r,K)$-controllable.
\end{theorem}

\subsection{A necessary condition for controllability}

We have seen in Theorem \ref{thm:main} that the linear control problem \eqref{eq:system} is $\alpha$-controllable on $d_{\Le}$-relatively dense sets. Here, we state a converse.

\begin{theorem} \label{thm:nessecary}
Let $D \subseteq X$, $T > 0$, $r\in [1,\infty)$, $K \geq 0$, $\alpha \in [0,\euler^{-\Dmax T})$, and assume that the linear control problem \eqref{eq:system} is $(\alpha , T , r , K)$-controllable. Then $D$ is $d_\comb$-relatively dense.
\end{theorem}

Note that in Theorem~\ref{thm:nessecary} we show relative density with respect to $d_{\comb}$, while in Theorem~\ref{thm:main} we assume density with respect to $d_{\Le}$.

\subsection{Consequences for exponential stabilization} 
We now turn to stabilizability and state some direct consequences of our result on cost-uniform $\alpha$-controllability.
\begin{definition}
    Let $D\subseteq X$ and $r\in[1,\infty]$. We say that the linear control problem \eqref{eq:system} is \emph{cost-uniformly open-loop stablizable with respect to $L_r$} if there exist $M\geq 1$, $\omega<0$ and $K\geq 0$ such that for all $f_0\in \ell_2(X,m)$ there exists $u\in L_r((0,\infty);\ell_2(D,m|_D))$ such that
    \[\norm{u}_{L_r((0,\infty);\ell_2(D,m|_D))} \leq K\norm{f_0}\qquad\text{and}\qquad \norm{f(t)} \leq M\euler^{\omega t} \norm{f_0},\quad t\geq 0.\]    
\end{definition}

In fact, cost-uniform $\alpha$-controllability with $\alpha\in [0,1)$ is equivalent to cost-uniform open-loop stabilizability, see \cite[Theorem 1]{TrelatWX-20} for the case $r=2$ and \cite[Proposition 1]{EgidiGST-24} for $r \in [1,\infty]$. We reformulate them as a proposition in our context, for a proof see \cite{TrelatWX-20,EgidiGST-24}.

\begin{proposition}
    Let $D\subseteq X$, $r \in[1,\infty]$. Then the following are equivalent.
    \begin{enumerate}
     \item[(i)]
     There exist $\alpha\in [0,1)$, $T>0$ and $K\geq 0$ such that the linear control problem \eqref{eq:system} is cost-uniformly $(\alpha,T,r,K)$-controllable.
     \item[(ii)]
     The linear control problem \eqref{eq:system} is cost-uniformly open-loop stabilizable with respect to $L_r$.
    \end{enumerate}
\end{proposition}

Thus, together with Theorem \ref{thm:main} we can now conclude cost-uniform open-loop stabilizability.

\begin{corollary}
    Let $D\varsubsetneq X$ be $d_{\Le}$-relatively dense, $r\in [1,\infty]$.
    Then the linear control problem \eqref{eq:system} is cost-uniformly open-loop stabilizable with respect to $L_r$.
\end{corollary}

\begin{definition}
    Let $D\subseteq X$. We say that the linear control problem \eqref{eq:system} is \emph{closed-loop stablizable} (also called \emph{exponential stabilizable}) if there exists a bounded linear operator $F\colon \ell_2(D,m|_D) \to \ell_2(X,m)$ such that $-H+\1_DF$ generates an exponentially stable $C_0$-semigroup $(S^F_t)_{t\geq 0}$, i.e.\ there exist $M\geq 1$ and $\omega<0$ such that 
    \[\norm{S^F_t} \leq M\euler^{\omega t},\quad t\geq 0.\]
\end{definition}
For closed-loop stabilizable control problems, the operator $F$ is called \emph{feedback operator} and the control function $u\in L_2((0,\infty);\ell_2(D,m|_D))$ is given by $u(t) := F(t) = FS_t f_0$ for all $t\geq 0$ and $f_0\in \ell_2(X,m)$.

In the Hilbert space case, i.e.\ $r=2$ in our situation, cost-uniform open-loop stabilizability also implies closed-loop stabilizability; cf.\ \cite[Theorem 4.3]{Zabczyk-08}. Thus, we obtain the next corollary.

\begin{corollary}
    Let $D\varsubsetneq X$ be $d_{\Le}$-relatively dense.
    Then the linear control problem \eqref{eq:system} with $r=2$ is closed-loop stabilizable.
\end{corollary}

\section{Proofs}
\label{sec:proof_obs_control}
\subsection{Proofs of Theorem \ref{thm:main} and Theorem \ref{thm:weak_obs}}

The controllability property in Definition~\ref{def:controllability} is equivalent to a weak observability estimate of the adjoint system. This equivalence has been established in \cite{TrelatWX-20} for linear control problems in Hilbert spaces and $r=2$, and subsequently generalized to Banach spaces in \cite{EgidiGST-24}. For a precise mathematical formulation, let us consider the adjoint system to \eqref{eq:system}; that is,
\begin{equation} \label{eq:AdjointSystem}
\begin{aligned}
  \dot{\varphi}(t) & = -H \varphi(t), \quad & & t>0  ,\quad \varphi(0)  = \varphi_0 \in \ell_2 (X,m), \\
  \psi(t)  &= \varphi(t)|_D, \quad   & & t \geq 0.
  \end{aligned}\tag{\text{Adj}$(H,D)$}
\end{equation}
\begin{definition}
Let $T > 0$, $\alpha , K \geq 0$, and $r \in [1,\infty]$. The adjoint problem \eqref{eq:AdjointSystem} satisfies a \emph{weak observability estimate at time $T$ with constants $\alpha$ and $K$ in $L_{r}$} if for all $\varphi_0 \in \ell_2 (X,m)$ we have
\[
 \lVert S_T \varphi_0 \rVert_{\ell_2 (X,m)} \leq K \lVert (S_{(\cdot)} \varphi_0)|_D  \rVert_{L_{r} ((0,T);\ell_2 (D,m|_D))} + \alpha \lVert \varphi_0 \rVert_{\ell_2 (X,m)}.
\]
We will use the short hand notation $(T,\alpha,K,r)$-observability estimate for this property.
\end{definition}
The following theorem is a special case of \cite[Theorem~1]{TrelatWX-20} if $r = 2$. If $r \not = 2$ it follows from \cite[Theorem~2.8]{EgidiGST-24}.
\begin{theorem}[\cite{TrelatWX-20,EgidiGST-24}]\label{thm:duality}
    Let $T>0$, $r,r' \in[1,\infty]$ such that $1/r+1/r' = 1$, and $\alpha,K \geq 0$. Then the following are equivalent.
    \begin{enumerate}
     \item[(i)]
     The linear control problem \eqref{eq:system} is $(\alpha,T,r,K)$-controllable.
     \item[(ii)]
     The adjoint problem \eqref{eq:AdjointSystem} satisfies an $(T,\alpha,K,r')$-observability estimate.
    \end{enumerate}
\end{theorem}

For a Borel set $I\subseteq \R$ we denote the spectral projection of $H$ by $P_I(H):=\1_{I}(H)$.
In \cite{LenzSS-20}, the following quantitative uncertainty principle for small energies was shown for the metric $d_{\Le}$.

\begin{proposition}[{\cite[Corollary 5.2]{LenzSS-20}}]
\label{prop:UP}
    Let $D\varsubsetneq X$ be $d_{\Le}$-relatively dense and $\Omega:=X\setminus D$. Let $I\subseteq \R$ be a Borel set such that $\sup I < \bigl(\Inr(\Omega)\vol(\Inr(\Omega))\bigr)^{-1}$. Then
    \[P_I(H) \leq \frac{16 \norm{H+1}^4}{\bigl(\bigl(\Inr(\Omega)\vol(\Inr(\Omega))\bigr)^{-1}-\sup I\bigr)^2} P_I(H)\1_D P_I(H).\]
\end{proposition}

\begin{remark}
 Using the results from \cite{ls}, (see also an announcement in \cite{p-owr}, in particular Theorem 1) and \cite{lss18}, in particular Corollary 3.7,  we obtain the following alternative bound, which uses a slightly more incisive energy condition but yields somewhat simpler bounds, in particular if $m$ is the counting measure:\\
 Let $I\subseteq \R$ be a Borel set such that
 $$\sup I \le \frac{\inf m}{42 \Inr(\Omega)\vol(\Inr(\Omega))^2}.$$ Then
    \[P_I(H) \leq \frac{42 \vol(\Inr(\Omega))}{\inf m} P_I(H)\1_D P_I(H).\]
\end{remark}

\begin{proof}[Proof of Theorem \ref{thm:weak_obs}]
    The proof is an application of \cite[Proposition 2]{EgidiGST-24} taking into account Proposition \ref{prop:UP}. We provide here the details for the readers convenience.
    
    Let $I:=(-\infty,\lambda]$ and $P_\lambda:=P_I(H)$. Let $\psi\in \ell_2(X,m)$. 
    We have \[\norm{P_\lambda S_\tau \psi} \leq \kappa \norm{(P_\lambda S_\tau \psi)|_D}_{\ell_2(D,m|_D)}\]
    for all $\tau \geq 0$ by Proposition \ref{prop:UP}.
    For $\tau\in [\delta T,T]$ we have
    \[\norm{(\id-P_\lambda)S_\tau \psi} \leq \euler^{-\tau \lambda} \norm{\psi}\]
    by the spectral theorem.
    Thus, for these $\tau$,
    \[\norm{S_\tau \psi} \leq \norm{P_\lambda S_\tau \psi} + \norm{(\id-P_\lambda)S_\tau \psi} \leq \kappa \norm{(P_\lambda S_\tau \psi)|_D}_{\ell_2(D,m|_D)} + \euler^{-\tau \lambda} \norm{\psi}.\]
    Further,
    \begin{align*}
        \norm{(P_\lambda S_\tau \psi)|_D}_{\ell_2(D,m|_D)} & \leq \norm{(S_\tau \psi)|_D}_{\ell_2(D,m|_D)} + \norm{((\id-P_\lambda) S_\tau \psi)|_D}_{\ell_2(D,m|_D)} \\
        & \leq \norm{(S_\tau \psi)|_D}_{\ell_2(D,m|_D)} + \norm{(\id-P_\lambda) S_\tau \psi},
    \end{align*}
    so
    \[\norm{S_\tau \psi}\leq \kappa \norm{(S_\tau \psi)|_D}_{\ell_2(D,m|_D)} + (\kappa+1) \euler^{-\tau \lambda} \norm{\psi}.\]
    Now, taking into account that $(S_t)$ is a contractive semigroup and that $\tau\geq \delta T$, we arrive at
    \[\norm{S_T\psi} \leq \norm{S_{T-\tau}}\norm{S_\tau \psi} \leq \kappa \norm{(S_\tau \psi)|_D}_{\ell_2(D,m|_D)} + (\kappa+1) \euler^{-\delta\lambda T} \norm{\psi}.\]
    Integrating with respect to $\tau\in[\delta T,T]$, we conclude
    \[\norm{S_T\psi} \leq \kappa \frac{1}{(1-\delta)T} \norm{(S_{(\cdot)} \psi)|_D}_{L_1((\delta T,T);\ell_2(D,m|_D))} + (\kappa+1) \euler^{-\delta\lambda T} \norm{\psi},\]
    which yields the assertion for $r=1$. An additional application of H\"older's inequality then yields the assertion for general $r$.
\end{proof}

\begin{proof}[Proof of Theorem \ref{thm:main}]
    We apply Theorem \ref{thm:weak_obs} with $\delta=1/2$, 
    $T = -\ln (\alpha / (\kappa+1)) / (\delta\lambda )>0$,
    and $r'\in[1,\infty]$ instead of $r$ and obtain that \eqref{eq:AdjointSystem} satisfies an $(T,\alpha,K,r')$-observability estimate. Now, Theorem \ref{thm:duality} yields that \eqref{eq:system} is $(\alpha,T,r,K)$-controllable.
\end{proof}

\begin{remark}
    In case of the continuum $\R^d$ instead of a graph it is well-known that the heat equation is $0$-controllable if and only if $D$ is a so-called thick set \cite{EgidiV-18,WangWZZ-19,BombachGST-23}. 
    This is a consequence of the fact that a quantitative uncertainty principle (sometimes also called spectral inequality or Carleman estimate) is valid for arbitrary large energies, together with a suitable dissipation estimate.
    Since in our case of heat equations on graphs we cannot achieve $0$-controllability in general by Theorem \ref{thm:non-null-controllability_Z}, in turn, we cannot expect to get a quantitative uncertainty principle as in Proposition \ref{prop:UP} for arbitrary energy intervals $I\subseteq \R$, in particular those which reach the upper bound of the spectrum of the Laplacian $H$, which fits to the dissipation estimate trivially coming from spectral theory. We refer to \cite{LebeauR-95,TenenbaumT-11,BeauchardP-18,NakicTTV-20,GallaunST-20,BombachGST-23a} for more on the abstract version of the so-called Lebeau--Robbiano method to show observability and, thus, $0$-controllability.
\end{remark}

\subsection{Proofs of Theorem \ref{thm:CoverNotControllable} and  Theorem \ref{thm:non-null-controllability_Z}}

We start with a lemma.

\begin{lemma}\label{lem:amenableCover}
Let $(X_1,b_1,m_1)$ be a finite connected graph and $(X_2,b_2,m_2)$ an amenable covering graph of $(X_1,b_1,m_1)$. Let $x_1 \in X_1$. Then, there exists a sequence $(Z_n)_{n\in\N}$ of finite subsets of $X_2$ such that
\[
\frac{b_2(Z_n,Z_n^c)}{m_2(Z_n\cap p^{-1}(\{x_1\}))} \to 0 \mbox{ as }n \to \infty.
\]
\end{lemma}
Here, $b_2(Z, Z^c) = \sum_{x\in Z}\sum_{y\in Z^c} b_2(x,y)$.

\begin{proof}
Let $d:=\sup_{x_1,y_1\in X_1} d_1(x_1,y_1)<\infty$ be the diameter of $(X_1,b_1,m_1)$ and let $(Y_n)_{n\in\N}$ be a sequence of finite subsets of $X_2$ such that
\[
\frac{b_2(Y_n,Y_n^c)}{m_2(Y_n)} \to 0 \mbox{ as }n \to \infty.
\]
We define $Z_n := B_d(Y_n) =\bigcup_{z\in Y_n} B_d(z)$ for $n\in\N$.
We claim that there exists $C\geq 0$ such that
\[
b_2(Z_n,Z_n^c) \leq C b_2(Y_n, Y_n^c)
\]
and
\[
m_2(Y_n) \leq C m_2(Z_n\cap p^{-1}(\{x_1\})),
\]
from which the lemma follows easily.

For the first inequality, we notice that the edge boundary is comparable with the inner vertex boundary, as the geometry is bounded.
Moreover, the inner vertex boundary of $Z_n$ is contained in the ball of radius $d$ around the inner vertex boundary of $Y_n$, and the volumes of balls of radius $d$ are uniformly bounded.

For the second inequality, we notice that every ball of radius $d$ contains at least one vertex of $p^{-1}(\{x_1\})$, and the volumes of the balls of radius $d$ around vertices of $p^{-1}(\{x_1\})$ are uniformly bounded. The inequality now follows by a double counting argument.

This finishes the proof. 
\end{proof}

\begin{proof}[Proof of Theorem \ref{thm:CoverNotControllable}]
For $M\subseteq X_2$ we write $1_M\from X_2\to\R$ with $1_M(x) = 1$ for $x\in M$ and $1_M(x)=0$ otherwise for the characteristic function of $M$, and for $x\in X_2$ we abbreviate $1_x:=1_{\{x\}}$.

For finite graphs such as $(X_1,b_1,m_1)$, it is well known that non-$0$-controllability of $(H_1,D_1)$ is equivalent to the existence of an eigenfunction of $H_1$ vanishing on $D_1$ (this is a verison of the so-called Hautus test, cf.\ \cite[Lemma 3.3.7]{Sontag-98}).

Let $\phi$ be an eigenfunction of $H_1$ to some eigenvalue $\lambda$, vanishing on $D_1$. We assume without obstruction that $\|\phi\|_\infty = 1$ and that the supremum norm is attained at some $x \in X_1$.
Let $\epsilon>0$. As $(X_2,b_2,m_2)$ is amenable and by Lemma~\ref{lem:amenableCover} there exists a finite subset $Z \subseteq X_2$ such that
\[
b_2(Z, Z^c) \leq \epsilon m_2(Z \cap p^{-1}(\{x\})).
\]

We define $f: X_2 \to \R$
\[
f := (\phi \circ p)\cdot 1_{Z}.
\]
%
%
As $\abs{f} \leq 1_Z$ the the semigroup $(\euler^{-tH_2})_{t\geq 0}$ is positivity preserving, 
we obtain for all $y \in X_2$
\[
|\euler^{-tH_2} f(y)| \leq \euler^{-tH_2} \abs{f}(y) \leq \euler^{-tH_2} 1_{Z}(y).
\]
Note that $H_2$ is a non-negative self-adjoint operator and hence, for all $p\in [1,\infty]$ the semigroup $(\euler^{-tH_2})_{t\geq0}$ extrapolates to $\ell_p(X_2,m_2)$.
We observe that for $t\geq 0$ we have 
\[\langle 1_Z, \euler^{-tH_2} 1_{Z^c}\rangle_{\ell_1,\ell_\infty} = \langle \euler^{-tH_2} 1_Z, 1_{Z^c}\rangle_{\ell_1,\ell_\infty} = \sum_{y\in Z^c} \euler^{-tH_2} 1_{Z}(y) m(y).\]
Moreover,
\begin{align*}
\sum_{y\in Z^c}
 \euler^{-tH_2} 1_{Z}(y) m(y)
&= \sum_{y\in Z^c} \int_0^t -H_2 \euler^{-sH_2} 1_{Z}(y) \,\drm s \cdot m(y)
\\
&= \int_0^t \sum_{z \in X_2,y \in Z^c} b_2(y,z) \left(\euler^{-sH_2} 1_{Z}(z) - \euler^{-sH_2} 1_{Z}(y)\right) \,\drm s \\
&= \int_0^t \sum_{z \in Z,y \in Z^c} b_2(y,z) \left(\euler^{-sH_2} 1_{Z}(z) - \euler^{-sH_2} 1_{Z}(y)\right) \,\drm s 
\\&\leq t  b_2(Z,Z^c),
\end{align*}
where in the third equality we use that
$
b_2(y,z) \left(\euler^{-sH_2} 1_{Z}(z) - \euler^{-sH_2} 1_{Z}(y)\right) 
$
is antisymmetric in $y$ and $z$.
Thus for all $t\geq 0$,
\[
\sum_{y\in Z^c}|\euler^{-tH_2} f(y)|m(y) \leq \sum_{y\in Z^c}
 \euler^{-tH_2} 1_{Z}(y) m(y) \leq t  b_2(Z,Z^c).
\]
Since $1_Z = 1-1_{Z^c}$ we obtain with $D_2 := p^{-1}(D_1)$ that
\[
\sum_{y\in D_2\cap Z} |\euler^{-tH_2} f(y)|m_2(y) = \sum_{y\in D_2\cap Z} |\euler^{-tH_2}( \phi \circ p)(y) - \euler^{-tH_2} h(y)|m_2(y)
\]
with $h := (\phi \circ p)\cdot 1_{Z^c} \in \ell_\infty(X_2,m_2)$. Note that $\norm{h}_\infty \leq 1$ and thus $\abs{h}\leq 1_{Z^c}$.

We now show that $H_2 (\phi \circ p) = \lambda \phi \circ p$.
For $x_2\in X_2$ we have
\begin{align*}
    \lambda \phi(p(x_2)) & = (H_1 \phi) (p(x_2))\\
    & = \frac{1}{m_1(p(x_2))} \sum_{y_1\in X_1} b_1(p(x_2),y_1) (\phi(p(x_2)) - \phi(y_1)).
\end{align*}
Since $X_2$ is a covering graph of $X_1$, we thus observe
\begin{align*}
    \lambda \phi(p(x_2)) & = \frac{1}{m_2(x_2)} \sum_{y_2\in X_2} b_2(x_2,y_2) (\phi(p(x_2)) - \phi(p(y_2))) = H_2(\phi\circ p)(x_2).
\end{align*}
Thus,  
$H_2 (\phi \circ p) = \lambda \phi \circ p$. Hence, $t\mapsto \euler^{-\lambda t} \phi \circ p$ solves the heat equation for $H_2$.
By stochastic completeness, bounded solutions to the heat equation are unique
 implying
$\euler^{-tH_2} (\phi \circ p) = \euler^{-\lambda t} \phi \circ p$ for all $t\geq 0$ and thus, $\sum_{z\in D_2} |\euler^{-tH_2}( \phi \circ p)(z)| = 0$  as $\phi = 0$ on $D_1$.
We estimate the remaining term
\[
\sum_{y\in D_2\cap Z} |\euler^{-tH_2} h(y)| m_2(y) \leq \sum_{y\in D_2\cap Z} \euler^{-tH_2} 1_{Z^c}(y) m_2(y) \leq \langle 1_Z, \euler^{-tH_2} 1_{Z^c}\rangle_{\ell_1,\ell_\infty} \leq t  b_2(Z,Z^c).
\]
In summary, we obtain
\begin{align*}
\sum_{y\in D_2} \abs{\euler^{-tH_2} f(y)} m_2(y) 
& \leq \sum_{y\in D_2\cap Z} \abs{\euler^{-tH_2} f(y)} m_2(y) + \sum_{y\in D_2\cap Z^c} \abs{\euler^{-tH_2} f(y)} m_2(y) \\
& \leq t b_2(Z,Z^c) + \sum_{y\in Z^c} \abs{\euler^{-tH_2} f(y)} m_2(y) \\
& \leq 2tb(Z,Z^c)
 \leq 2 t \epsilon m_2(Z \cap p^{-1}(\{x\})).
\end{align*}
Moreover, we have $m_2(Z \cap p^{-1}(\{x\})) \leq \|f\|^2$ as $|f|=1$ on $Z \cap p^{-1}(\{x\})$.
As $|\euler^{-tH_2} f|\leq 1$, we get
\[
\|(\euler^{-tH_2} f)|_{D_2}\|_{\ell_2(D_2,m_2|_{D_2})}^2 =
\sum_{y\in D_2} \abs{\euler^{-tH_2} f(y)}^2 m_2(y) \leq \sum_{y\in D_2} \abs{\euler^{-tH_2} f(y)} m_2(y)\leq 2t\epsilon \|f\|^2.
\]
Let now, $T>0$, $r\in [1,\infty]$ and $K\geq 0$.
As the Laplacian $H_2$ is bounded,
\[
\|f\|^2 \leq \|\euler^{-T H_2} f\|^2 \exp(2T \|H_2\|).
\]
Thus, when we combine these two estimates, we obtain
\begin{align*}
    \norm{(\euler^{-tH_2} f)|_{D_2}}_{\ell_2(D_2, m_2|_{D_2})} & = \left(\norm{(\euler^{-tH_2} f)|_{D_2}}_{\ell_2(D_2, m_2|_{D_2})}^2\right)^{1/2} \\
    & \leq 2^{1/2} t^{1/2} \epsilon^{1/2} \exp(T\norm{H_2}) \norm{\euler^{-TH_2} f}.
\end{align*}
Thus, for the $L_r$-norms we observe
\[\norm{(\euler^{-(\cdot)H_2} f)|_{D_2}}_{L_r((0,T);\ell_2(D_2, m_2|_{D_2}))} \leq 2^{1/2} \left(\frac{2}{r+2}\right)^{1/r} T^{1/2+1/r} \epsilon^{1/2} \exp(T\norm{H_2}) \norm{\euler^{-TH_2} f}.\]
Here, $(2 / (r+2))^{1/r} = 1$ for $r=\infty$ which is consistent with taking the limit as $r\to\infty$.
%
Since $\epsilon>0$ is arbitrary and in view of Theorem \ref{thm:duality} we obtain that $(H_2,D_2)$ is not $(0,T,r,K)$-controllable.
\end{proof}

\begin{proof}[Proof of Theorem \ref{thm:non-null-controllability_Z}]
We observe that $\Z$ with the given $b$ is amenable and satisfies \eqref{eq:C}, \eqref{eq:B}, \eqref{eq:M}. Moreover, it is the covering graph of the connected finite 4-cycle $X_1:=\Z / 4\Z$ with $m_1:=1$, $b_1:=b|_{\{0,1,2,3\}^2}$ and $p:\Z\to \Z/4\Z$ given by $p(x):= x\mod 4$. Note that with $D_1 := \{0,2\}\subseteq \Z/4\Z$ we have $D=2\Z = p^{-1}(D_1)$.
As $\phi:\Z/4\Z\to \R$, 
\[\phi(x):=\begin{cases} 0 & x\in \{0,2\},\\
            1 & x=1,\\
            -1 & x=3
           \end{cases}\]
is an eigenfunction for $H_1$ vanishing on $D_1$, the linear control system $(H_1, D_1)$ is not $0$-controllable. Now, Theorem~\ref{thm:CoverNotControllable} yields the assertion for the control system $(H, D)$.
\end{proof}

\subsection{Proof of Theorem~\ref{thm:nessecary}}

The proof of Theorem~\ref{thm:nessecary} follows from the following two propositions. For $x \in X$ we denote by $\delta_x \in \ell^2 (X,m)$ the Dirac delta function given by 
\[
\delta_x (y) := 
    \begin{cases} 
    1/\sqrt{m(x)} & \text{if}\ y=x,\\ 
    0 & \text{if}\  y\not = x.
    \end{cases}
\]
Moreover, for $x \in X$ we use the notation $1_x = \sqrt{m(x)} \delta_x$ as above.

\begin{proposition}
\label{prop:lhs}
    For all $t \geq 0$ and all $x \in X$ we have
    \[
    \lVert S_t \delta_x \rVert \geq \euler^{- \Deg (x) t} .
    \]
\end{proposition}
\begin{proposition}
\label{prop:rhs}
    Let $D \subseteq X$ be not $d_{\comb}$-relatively dense.
    Then there exists a sequence $(x_n)_{n \in \N}$ in $X$ such that for all $T>0$ and all $r\in [1,\infty]$ we have 
    \[
    \lim_{n \to \infty}  \norm{(S_{(\cdot)} \delta_{x_n})|_D}_{L_r((0,T);\ell_2(D,m|_D))} = 0 .
    \]
\end{proposition}

For the proofs of Propositions~\ref{prop:lhs} and \ref{prop:rhs} we will employ the stochastic interpretation of the heat semigroup.
To this end, we denote by $(Z_t)_{t \geq 0}$ the (continuous time) Markov process associated to the semigroup $(S_t)_{t \geq 0}$, which can be constructed as follows, see, e.g., \cite[Section 2.5]{KellerLW-21} for details. 
Over a probability space $(\Omega , \mathcal{A} , \PP)$ we define a (discrete time) Markov process $Y=(Y_n)_{n \in \N_0}$ in $X$ with an initial distribution and transition probabilities
\begin{equation} \label{eq:transition}
 \PP (Y_1 = y_1 \mid Y_0 = y_0) = \frac{b(y_0,y_1)}{m(y_0) \Deg (y_0)}, \quad y_0,y_1 \in X ,
\end{equation}
as well as a sequence $(\xi_n)_{n \in \N}$ of independent exponentially distributed random variables with parameter 1 which is independent of $Y$ as well. Note that by assumption we have $b(x,x) = 0$ for all $x \in X$. In order to define the Markov process $Z = (Z_t)_{t \geq 0}$ 
we define the sequence of \emph{holding times} $(H_k)_{k \in \N}$ by $H_k := \xi_k / \Deg(Y_{k-1})$, and the sequence of \emph{jump times} $(J_k)_{k \in \N_0}$ by $J_0 := 0$ and $J_k := H_1 + H_2 + \ldots + H_k$ for $k \in \N$. The continuous time Markov process $Z = (Z_t)_{t \geq 0}$ is defined by $Z_t := Y_k$ if $t \in [J_k , J_{k+1})$. Roughly speaking, the Markov process $X$ stays at a vertex for an exponential distributed holding time, and then jumps at the jump time to another vertex according to the law \eqref{eq:transition}.

For $x \in X$ we denote by $\PP_x$ the probability measure on $\mathcal{A}$ given by $\PP_x (\cdot) = \PP (\cdot \mid X_0 = x)$, and by $\EE_x$ the expectation with respect to $\PP_x$. Then, \cite[Theorem 2.31]{KellerLW-21} implies for all $t\geq0$, $f \in \ell^2 (X,m)$ and $x\in X$
\begin{equation} \label{eq:FK}
 (S_t f)(x) = \EE_x (f (Z_t) ) .
\end{equation}
Note that in \cite[Theorem~3.31]{KellerLW-21} the expectation on the right hand side of \eqref{eq:FK} includes the additional characteristic function $\chi_{\{t < \zeta\}}$, where $\zeta$ is the so-called lifetime or explosion time of the process. However, in our situation with Assumption \eqref{eq:B} we have that $(X,b,m)$ is stochastically complete by \cite[Theorem 4.3]{Wojciechowski-21} and hence $\zeta = \infty$ with probability one by \cite[Theorem 7.32]{KellerLW-21}.

\begin{proof}[Proof of Proposition~\ref{prop:lhs}]
Fix $t \geq 0$ and $x \in X$. Then we have
\begin{equation} \label{eq:lower1}
    \norm{S_t \delta_x}^2  = \langle S_t \delta_x, S_t\delta_x\rangle = \langle S_{2t} \delta_x, \delta_x\rangle
    = \sum_{y\in X} (S_{2t}\delta_x) (y) \delta_x(y) m(y) = (S_{2t}1_x)(x) ,
\end{equation}
where we used that $S_t$ is a self-adjoint. By \eqref{eq:FK} we find 
\begin{equation} \label{eq:lower2}
 (S_{2t} 1_x)(x) =  \EE_x ( 1_x (Z_{2t}) ) 
    = \PP_x (X_{2t} = x) \geq \PP_x (H_1 > 2t) .
\end{equation}
Recall that $H_1 = \xi_1 / \Deg (Y_0) = \xi_1 / \Deg (Z_0)$, where $\xi_1$ is exponentially distributed with parameter 1 and independent of $Y_0 = Z_0$. Hence, we have
\begin{equation} \label{eq:lower3}
   \PP_x (H_1 > 2t) = \PP (\xi_1 > 2 \Deg (Z_0) t  \mid Z_0 = x) = \PP (\xi_1 > 2 \Deg (x) t ) = \euler^{- 2 \Deg (x) t} . 
\end{equation}
The claim now follows from \eqref{eq:lower1}, \eqref{eq:lower2}, and \eqref{eq:lower3}.
\end{proof}
\begin{proof}[Proof of Proposition~\ref{prop:rhs}]
First we explain the construction of the sequence $(x_n)_{n \in \N}$ in $X$.  Since $D$ is not $d_\comb$-relatively dense, we have
$\Covr (D) = \infty$.
With other words, for all $R > 0$ we have (with respect to the metric $d_\comb$)
\[
 \bigcup_{x\in D} B_R (x) \not =  X .
\]
For each $n \in \N$ we choose $x_n \in X$ such that
$d_{\comb} (x_n , D) = \inf_{x \in D} d_{\comb}(x_n,x) \geq n$.
\par
Note that $0\leq S_t 1_x\leq 1$ for all $t\geq 0$ and $x\in X$ by \eqref{eq:FK}.
Since $ ( S_t 1_{x_n})(y) m(y) =  ( S_t 1_{y})(x_n) m (x_n)$, we have for all $n \in \N$
\begin{equation} \label{eq:NormD}
\lVert (S_t \delta_{x_n})|_D \rVert^2  
    = \sum_{y \in D} ( S_t 1_{x_n})(y)^2 \frac{m(y)}{m(x_n)} 
    \leq \sum_{y \in D} ( S_t 1_{x_n})(y) \frac{m(y)}{m(x_n)}
    =  \sum_{y \in D} ( S_t 1_{y})(x_n).
\end{equation}
For the last sum we use \eqref{eq:FK} and calculate
\begin{align}
\sum_{y \in D} ( S_t 1_{y})(x_n)   = \sum_{y \in D} \EE_{x_n} (1_y (Z_t)) 
&\leq  \sum_{y \not \in B_{n-1}(x_n)} \EE_{x_n} (1_y (Z_t)) \nonumber \\ 
&= \sum_{y \not \in B_{n-1}(x_n)} \PP_{x_n} (Z_t = y) = \PP_{x_n} (Z_t \not \in B_{n-1} (x_n)) . \label{eq:SumD}
\end{align}
The last probability can be estimated from above by the probability that the Markov process does at least $n$ jumps within the time interval $[0,t]$, or with other words, that the sum of the first $n$ holding times is less or equal to $t$, i.e. for all $n \in \N$ we have
\[
 \PP_{x_n} (Z_t \not \in B_{n-1} (x_n))
 \leq 
 \PP_{x_n} \left( \sum_{k=1}^n H_k \leq t \right) .
\]
For $x \in X$ and $n \in \N$ we denote by 
\[
 \Gamma (x,n) := \left\{\gamma = (\gamma_0 , \gamma_1 , \ldots , \gamma_n) :\; \gamma \,\text{path in $X$},\, \gamma_0 = x\right\}
\]
the set of all paths of length $n$ starting in $x$ with the convention that $\Gamma(x,0) = \{x\}$, and by $\{Y_{[0,n]} = \gamma\}$ the event $\{Y_0=\gamma_0,Y_1=\gamma_1,\ldots , Y_n=\gamma_n\}$. Then, by the law of total probability, we have
\begin{align*}
 \PP_{x_n} \left( \sum_{k=1}^n H_k \leq t \right)
 &= \sum_{\gamma \in \Gamma (x_n , n-1)} \!\!\!\!\! \PP_{x_n} \left( \sum_{k=1}^n \xi_k \leq \Deg(Y_{k-1}) t \mid Y_{[0,n-1]}=\gamma \right) \PP_{x_n} \left(Y_{[0,n-1]}=\gamma \right) \\
 &= \sum_{\gamma \in \Gamma (x_n , n-1)} \!\!\!\!\! \PP_{x_n} \left( \sum_{k=1}^n \xi_k \leq \Deg (\gamma_{k-1})t \mid Y_{[0,n-1]}=\gamma \right) \PP_{x_n} \left(Y_{[0,n-1]}=\gamma \right) .
\end{align*}
Since $(\xi_k)_{k \in \N}$ is independent of $Y$, and since $\Dmax = \sup_{x \in X} \Deg (x) < \infty$ by Assumption \eqref{eq:B}, we find
\[
 \PP_{x_n} \left( \sum_{k=1}^n H_k \leq t \right) 
 \leq
 \sum_{\gamma \in \Gamma (x_n , n-1)} \!\!\!\!\! \PP_{x_n} \left( \sum_{k=1}^n \xi_k \leq \Dmax t \right) \PP_{x_n} \left(Y_{[0,n-1]}=\gamma \right) .
\]
As the sum of $n$ independent exponentially distristributed random variables with parameter $1$ is Erlang distributed with parameters $n$ and $1$, we conclude
\[
 \PP_{x_n} \Bigl( \sum_{k=1}^n H_k \leq t \Bigr) 
 \leq
 \mathrm {e}^{-\Dmax t}\sum _{i=n}^{\infty}{\frac {(\Dmax t)^{i}}{i!}} \!\!\!\!\! \sum_{\gamma \in \Gamma (x_n , n-1)} \!\!\!\!\! \PP_{x_n} \left(Y_{[0,n-1]}=\gamma \right) 
 = \mathrm {e}^{-\Dmax t}\sum _{i=n}^{\infty}{\frac {(\Dmax t)^{i}}{i!}} ,
\]
and hence
\begin{align}
 \PP_{x_n} (Z_t \not \in B_{n-1} (x_n))
 &\leq \mathrm {e}^{-\Dmax t}\sum _{i=n}^{\infty}{\frac {(\Dmax t)^{i}}{i!}} . \label{eq:ProbD}
\end{align}
For all $n \in \N$ we conclude from \eqref{eq:NormD}, \eqref{eq:SumD} and \eqref{eq:ProbD} that
\[
\lVert (S_t \delta_{x_n})|_D \rVert^2   \leq \mathrm {e}^{-\Dmax t}\sum _{i=n}^{\infty}{\frac {(\Dmax t)^{i}}{i!}} .
\]
Thus,
\[\norm{(S_t \delta_{x_n})|_D} \leq \mathrm {e}^{-\frac{1}{2}\Dmax t}\left(\sum _{i=n}^{\infty}{\frac {(\Dmax t)^{i}}{i!}}\right)^{1/2}.\]
Now, let $T>0$ and $r\in [1,\infty]$. Then for the supremum of the right-hand side with respect to $t\in [0,T]$ we estimate
\[\sup_{t\in [0,T]} \mathrm {e}^{-\frac{1}{2}\Dmax t}\left(\sum _{i=n}^{\infty}{\frac {(\Dmax t)^{i}}{i!}}\right)^{1/2} = \left(\sum _{i=n}^{\infty}{\frac {(\Dmax T)^{i}}{i!}}\right)^{1/2} \to 0\]
as $n\to \infty$. Thus, also
\[
\norm{(S_{(\cdot)} \delta_{x_n})|_D}_{L_r((0,T);\ell_2(D,m|_D))} \to 0. \qedhere
\]
\end{proof}

\begin{proof}[Proof of Theorem \ref{thm:nessecary}]
We proof the theorem by contraposition. Assume that $D \subseteq X$ is not $d_\comb$-relatively dense. Let $T>0$  and $r\in [1,\infty]$, and let $r'\in [1,\infty]$ be the H\"older conjugate of $r$.
By Proposition~\ref{prop:lhs} and Assumption~\eqref{eq:B} we have for all $x \in X$
\[
    \lVert S_T\delta_x \rVert \geq \euler^{-\Deg (x) T} \geq \euler^{- \Dmax T}.
\]
By Proposition~\ref{prop:rhs} there exists a sequence $(x_n)_{n \in \N}$ in $X$ such that
\[
    \lim_{n \to \infty} \norm{(S_{(\cdot)} \delta_{x_n})|_D}_{L_{r'}((0,T);\ell_2(D,m|_D))} = 0 .
\]
Thus, for all $\alpha \in [0,\euler^{-\Dmax T})$ and all $K \geq 0$ there exists $n \in \N$ such that 
\[
    \lVert S_T\delta_{x_n} \rVert  \geq \euler^{-\Dmax T}  > K \norm{(S_{(\cdot)} \delta_{x_n})|_D}_{L_{r'}((0,T);\ell_2(D,m|_D))} + \alpha \lVert \delta_{x_n} \rVert,
\]
since $\norm{\delta_{x}} = 1$ for all $x\in X$.
Theorem~\ref{thm:duality} implies that for all $\alpha \in [0,\euler^{-\Dmax T})$ and all $K \geq 0$, the linear control problem \eqref{eq:system} is not $(\alpha , T , r , K)$-controllable.
\end{proof}

\printbibliography

\end{document}